\documentclass[12pt,reqno]{amsart}
\usepackage{amsfonts,amsmath,amscd,amssymb,amsthm,color}
\usepackage{enumerate,hyperref,perpage,url}
\usepackage[margin=2.3cm]{geometry}

\theoremstyle{plain}
\newtheorem{thm}{Theorem}

\newtheorem{cor}[thm]{Corollary}
\newtheorem{lemma}[thm]{Lemma}
\newtheorem{prop}[thm]{Proposition}

\theoremstyle{remark}
\newtheorem{defn}[thm]{Definition}

\newtheorem*{rmk}{\textbf{Remark}}

\numberwithin{equation}{section}
\numberwithin{thm}{section}

\MakePerPage{footnote} \allowdisplaybreaks 

\newcommand{\norm}[1]{\left|\!\left|{#1}\right|\!\right|}

\newcommand{\dist}{\mathrm{dist}}
\newcommand{\Exp}{\mathrm{exp}}
\newcommand\E{\mathbb E}

\newcommand\HH{\mathcal H}
\newcommand\Inj{\mathrm{Inj}\,}
\newcommand\Loop{\mathcal L}
\newcommand{\Lip}{\mathrm{Lip}}
\newcommand\M{\mathbb M}
\newcommand\Me{\mathcal M}

\newcommand{\Planck}{\mathrm{Planck}}

\newcommand\R{\mathbb R}
\newcommand{\s}{\mathbb S}

\newcommand{\Span}{\mathrm{span}\,}

\newcommand\Id{\text{Id}}

\newcommand{\Var}{\mathrm{Var}}
\newcommand{\ve}{\varepsilon}

\newcommand{\Vol}{\mathrm{Vol}}

\title[Equidistribution of random waves on small balls]{Equidistribution of random waves on small balls}

\author{Xiaolong Han}
\address{Department of Mathematics, California State University, Northridge, CA 91330, USA}
\email{xiaolong.han@csun.edu}

\author{Melissa Tacy}
\address{Department of Mathematics and Statistics, University of Otago, Dunedin Otago 9016, New Zealand }
\email{mtacy@maths.otago.ac.nz}



\date{}

\begin{document}
\maketitle

\begin{abstract}
In this paper, we investigate the small scale equidistribution properties of randomised sums of Laplacian eigenfunctions (i.e. random waves) on a compact manifold. We prove small scale expectation and variance results for random waves on all compact manifolds. Here, ``small scale'' refers to balls of radius $r(\lambda)\to 0$ such that $r/r_{\text{Planck}}\to\infty$, where $r_{\text{Planck}}$ is the Planck scale. For balls at a larger scale (although still $r(\lambda)\to 0$) we also obtain estimates showing that the probability that a random wave fails to equidistribute decays exponentially with the eigenvalue.  
\end{abstract}

\section{Introduction}
Studying the behaviour of random combinations of either plane waves or eigenfunctions has lately proved to be an exciting research area. It is conjectured, by Berry \cite{B} in the 1970s, that eigenfunctions of chaotic systems such as planar domains with chaotic billiard flow behave like random waves. That is, their behaviour is modelled by functions of the form
\begin{equation}\label{eq:randomwaveRn}
\sum_{j}a_{j}e^{i\lambda\langle x,\xi_{j}\rangle},
\end{equation}
where the $\{\xi_{j}\}$ are chosen as a set of equidistributed (at scale $\lambda^{-1}$) directions on the $(n-1)$-$\dim$ unit sphere $\s^{n-1}$ and the coefficients $a_{j}$ are chosen randomly, e.g. independent Gaussian random variables. 

In the setting of manifolds where the underlying geodesic flow displays chaotic properties equidistribution for Laplacian eigenfunctions has been studied by Snirelman \cite{Sn}, Zelditch \cite{Z1}, Colin de Verdi\`ere \cite{CdV}. In particular, on manifolds with ergodic geodesic flow such as negatively curved manifolds, there is a full density sequence of eigenfunctions in any eigenbasis that equidistribute. Recently a number of small scale equidistribution results were established in various settings and at various scales by Han \cite{Ha1, Ha2}, Hezari-Rivi\`ere \cite{HR1, HR2}, Lester-Rudnick \cite{LR}, Chang-Zelditch \cite{CZ}. Notably on manifolds with negative curvature, equidistribution at logarithmic scales (i.e. $r=(\log\lambda)^{-\alpha}$ for some $\alpha>0$) is established in \cite{Ha1} and \cite{HR1}.  See also the application of small scale equidistribution to other eigenfunction problems in Hezari \cite{He1, He2, He3}, Sogge \cite{So1, So2}, Zelditch \cite{Z5}, etc.

As in the Euclidean case it is interesting to study the model case of random behaviour. On a compact manifold $\M$, the natural class of objects that replace plane waves $e^{i\lambda\langle x,\xi_{j}\rangle}$ are eigenfunctions. That is, we consider sums
$$\sum_{\lambda_{j}\in\Lambda}a_{j}e_{j}(x),$$
where $\Lambda\subset[0,\infty)$, the $e_{j}$'s are orthonormal Laplacian eigenfunctions on $\M$ with eigenvalues $\lambda_j^2$, and the coefficients $a_{j}$ are prescribed in a random fashion. 

The obvious first question is how to pick the set $\Lambda$. Notice that, because $\xi_j\in\s^{n-1}$, the plane waves $e^{i\lambda\langle x,\xi_j\rangle}$ are generalised eigenfunctions of the Laplacian with eigenvalues $\lambda^2$. So on compact manfolds, initially it may seem natural to fix an eigenspace $E_{\lambda}$ with eigenvalue $\lambda^2$ and randomise only over the eigenfunctions in $E_\lambda$, as is done in \cite{Ha2} on manifolds including tori and spheres. See also de Courcy-Ireland \cite{CI} on spheres. However, the multiplicity of this eigenvalue may be low. In fact, in chaotic cases such as when $\M$ has negative curvature it is conjectured that the eigenvalues have very low multiplicity. Indeed, on manifolds with generic metric, the eigenvalues are simple due to Uhlenbeck \cite{U}. Therefore, to capture the random behaviour, we allow ourselves to randomise over eigenfunctions whose eigenvalues sit in a spectral window. Such randomisations were introduced in Zelditch \cite{Z3}. That is, we set $\Lambda=[\lambda-W,\lambda]$ for $1\le W\le\lambda$ and consider the functions 
$$u=\sum_{\lambda_{j}\in[\lambda-W,\lambda]}{a_{j}}e_{j}(x).$$
We point out that the spectral window width $W$ here is allowed to depend on $\lambda$. These randomised functions are commonly referred as ``random waves''. We adopt this terminology and reserve the term ``random eigenfunctions'' for those combinations taken over a single eigenspace. In Zelditch \cite{Z3}, the two special cases of $W=1$ and $W=\lambda$ are called the asymptotically fixed frequency ensembles and the cut-off ensembles, respectively. Both of these ensembles are included in our analysis here.

Having selected a window to randomise over, we must now consider how we pick our random variables $a_j$ in the random wave $u=\sum a_je_j$. We always normalize such that $\E\left(\|u\|^2_{L^2(\M)}\right)=1$. Some common choices of $a_j$ include independent random variables such as Gaussian or Rademacher random variables with proper normalization (\cite{B, CI, Z3}) and uniform probability density on unit spheres (\cite{BL, Ha2, M, Z2, Z4}.) The coefficients $a_{j}$ in the former randomisation procedure are independent so it is straightforward to compute some probabilistic estimates such as the covariance. The latter randomisation process is not independent however in Section \ref{sec:prob} we show, in a rather elementary way, that these key probabilistic estimates are asymptotically the same as the independent case. Where the coefficients $a_{j}$ are chosen so that $a=(a_{1},\dots,a_{d})$ lies on the unit sphere $\s^{d}$ the randomisation procedure admits a dual interpretation, namely
\begin{itemize}
\item we pick coefficients $a_j$ at random from a probability density,
\item we pick an $L^{2}$ normalised function at random in $\Span_{\lambda_{j}\in[\lambda-W,\lambda]}\{e_j\}$. (Note that here $\|u\|^2_{L^2(\M)}=\sum|a_j|^2=1$ for all random waves, whereas one can only require $\E\left(\|u\|^2_{L^2(\M)}\right)=1$ in the case when $a_j$ are chosen as Gaussian independent variables.)
\end{itemize}
In addition, the Levy concentration of measure (Theorem \ref{thm:Levy}) on the unit sphere serves as an important tool to study uniform equidistribution of random waves on the whole manifold (see Theorem \ref{thm:SSEonM}.) Therefore, in this paper, we use the uniform probability density on unit spheres and ask about the expected behaviour of random waves as well as the variance in behaviour. In particular, we focus on small scale behaviour. We want to understand when random waves equidistribute on small balls. 

There are two parts to understanding this equidistribution. The first is to ascertain when 
\begin{equation}\label{expectation-equi}
\E\left(\int_{B(x,r)}|u|^{2}\,d\Vol\right)\to{}\frac{\Vol(B(x,r))}{\Vol(\M)}\quad\text{as }\lambda\to\infty,
\end{equation}
in which $B(x,r)\subset\M$ is a geodesic ball with center $x$ and radius $r$. However, while the expectation value might equidistribute, it is still possible that the probability of non-equidistribution is high. To that end we also determine, for given $r=r(\lambda)\to0$ as $\lambda\to\infty$, whether
\begin{equation}\label{var-equi}
\Var\left(\int_{B(x,r)}|u|^{2}d\Vol\right)=o\left(\Vol(x,r)^{2}\right)\quad\text{as }\lambda\to\infty.
\end{equation}
The variance estimate tells us whether we may expect that a typical sum of eigenfunctions equidistributes at scale $r$ around $x\in\M$. If the scale $r$ in \eqref{expectation-equi} and \eqref{var-equi} is independent of $\lambda$, i.e. equidistribution at fixed scales, the analysis follows from Zelditch \cite{Z2, Z4} and Maple \cite{M} in various random settings.

At small scales $r=r(\lambda)$ such that $r^{-1}=o(\lambda)$ as $\lambda\to\infty$, we are able to obtain sufficient conditions for \eqref{expectation-equi} and \eqref{var-equi}. Notice that these scales are just above the Planck scale
$$r_\Planck:=\lambda^{-1}.$$
Precisely, the scale $r$ satisfies that $r/r_\Planck\to\infty$. See Theorems \ref{thm:sserandom}   and  \ref{thm:sserandomfixed} below.

For a \textit{fixed} $x\in\M$, \eqref{expectation-equi} and \eqref{var-equi} tell us that probability such that the random wave $u$ does not equidistribute on $B(x,r)$ decays as $\lambda\to\infty$. Finally, we consider small scale equidistribution of random waves \textit{uniformly} on the manifold, i.e. uniformity across balls $B(x,r)$ of radius $r$ and for all $x\in\M$. See Theorem \ref{thm:SSEonM} below.

At this point, some details of  our setup  is in order. Let $(\M,g)$ be an $n$-$\dim$ compact, smooth Riemannian manifold without boundary. Denote $\Delta=\Delta_g$ the (positive) Laplace-Beltrami operator. Let $\{e_j\}_{j=0}^\infty$ be an orthonormal basis of eigenfunctions (i.e. eigenbasis) of $\Delta$ with eigenvalues $\lambda_j^2$ (counting multiplicities). That is, $\Delta e_j=\lambda^2_je_j$, where $\lambda_j$ is called the eigenfrequency. Denote $\Inj\M$ the injectivity radius of $\M$. We assume, without loss of generality, that $\Inj\M\ge1$.

We define the probability space for the random waves in a similar fashion as in Zelditch \cite{Z3}. 
\begin{defn}[Random waves]\label{def:probspace}
Let
\begin{equation}\label{HWdef}
\HH_{W}(\lambda)=\Span_{\lambda_{j}\in[\lambda-W,\lambda]}\{e_j\}\quad\text{and}\quad N_{W}(\lambda)=\dim\HH_W(\lambda).
\end{equation}
We assume that the eigenfunctions $e_j$ are real-valued. We define the random wave $u_{\lambda,a}\in\HH_{W}(\lambda)$ as
\begin{equation}\label{eq:ulambda}
u_{\lambda,a}:=\sum_{\lambda_{j}\in[\lambda-W,\lambda]}a_je_j,\quad\text{for }a\in\s^{N_W(\lambda)-1}.
\end{equation}
Here, $\s^{N_W(\lambda)-1}$ is equipped with uniform probability measure $\mu_{N_W(\lambda)}$. That is, $u_{\lambda,a}$ is a sum of eigenfunctions in $\HH_W(\lambda)$ with random coefficient $a\in\s^{N_W(\lambda)-1}$ so $\|u_{\lambda,a}\|_{L^2(\M)}=1$. For brevity, we also write $u_\lambda$ as $u_{\lambda,a}$ with the understanding that $a$ is the random variable.
\end{defn}

\begin{rmk}
One can similarly consider random waves as combinations of complex-valued Laplacian eigenfunctions, in which case the coefficient $a$ in \eqref{eq:ulambda} is chosen randomly from the complex unit sphere. The analysis is similar so we omit details here. 
\end{rmk}

Our first main theorem states that
\begin{thm}\label{thm:sserandom}
Let $1\le W\le\lambda$ and $x\in\M$. 
\begin{enumerate}[(i).]
\item For $r>0$, the expected value with respect to the probability measure $\mu_{N_W(\lambda)}$  
\begin{equation}\label{eq:sserandomEx}
\E\left(\int_{B(x,r)}|u_\lambda|^2\,d\Vol\right)=\frac{\Vol(B(x,r))}{\Vol(\M)}\left[1+O\left(W^{-1}\right)\right],
\end{equation}
where the term $O\left(W^{-1}\right)$ is independent of $x$.
\item For $r^{-1}=o(\lambda)$ as $\lambda\to\infty$, i.e. $r/r_\Planck\to\infty$, we have that the variance with respect to the probability measure $\mu_{N_W(\lambda)}$
\begin{equation}\label{eq:sserandomVar}
\Var\left(\int_{B(x,r)}|u_\lambda|^2\,d\Vol\right)=\Vol(B(x,r))^2\left[o(1)+O\left(W^{-2}\right)\right]\quad\text{as }\lambda\to\infty,
\end{equation}
where the terms $o(1)$ and $O\left(W^{-2}\right)$ are independent of $x$.
\end{enumerate}
In particular, if the spectral window width $W=W(\lambda)\to\infty$ as $\lambda\to\infty$, then  at all scales $r$ such that $r/r_\Planck\to\infty$,
$$\E\left(\int_{B(x,r)}|u_\lambda|^2\,d\Vol\right)=\frac{\Vol(B(x,r))}{\Vol(\M)}+o\left(r^n\right)\quad\text{as }\lambda\to\infty,$$
and
$$\Var\left(\int_{B(x,r)}|u_\lambda|^2\,d\Vol\right)=o\left(r^{2n}\right)\quad\text{as }\lambda\to\infty.$$
\end{thm}

We remark that if the spectral window width $W$ is independent of $\lambda$, e.g. the case of the asymptotically fixed frequency ensemble when $W=1$, then Theorem \ref{thm:sserandom} does \textit{not} imply small scale equidistribution results of the random waves $u_\lambda$. Instead, according to \eqref{eq:sserandomEx} and \eqref{eq:sserandomVar}, we can only conclude that the $L^{2}$ integral of $u_\lambda$ on the ball $B(x,r)$ is proportional to the normalised volume of the ball.

However, with a geometric condition on the manifold $\M$, we are able to recover small scale equidistribution results. The relevant condition is the set of the geodesic loop directions
$$\Loop_x:=\{\xi\in S^*\M:G_t(x,\xi)=(x,\eta)\text{ for some }t>0\text{ and }\eta\in S_x^*\M\}$$
is of measure zero in $S^*_x\M$ for all $x\in\M$. Here, $S_x^*\M$ is the cosphere space of $\M$ at $x$, $S^*\M$ is the cosphere bundle of $\M$ and $G_{t}(x,\xi)$ is the geodesic flow on $\M$. This pointwise aperodic condition is called the non self-focal condition. Examples of manifolds satisfying the non self-focal condition include the negatively curved manifolds (i.e. all sectional curvatures are negative everywhere.) Since manifolds with negative curvature are a key class of manifolds that we wish to understand using randomisation, making such an assumption is not as restrictive as may first appear. 

The above non self-focal condition is a natural dynamical condition to study the precise behavior of eigenfunctions restricted to a fixed-width spectrum window. See Section \ref{sec:spec} for the background. 

Concerning the small scale equidistribution of random waves in fixed-width spectral windows, we prove that 
\begin{thm}[Small scale equidistribution of random waves in fixed-width spectral windows]\label{thm:sserandomfixed}
Suppose that $W\ge1$ is independent of $\lambda$. Assume that the set of loop directions $\Loop_x$ is of measure zero in $S_x^*\M$ for all $x\in\M$. 
\begin{enumerate}[(i).]
\item The expectation
$$\E\left(\int_{B(x,r)}|u_\lambda|^2\,d\Vol\right)=\frac{\Vol(B(x,r))}{\Vol(\M)}+o\left(r^n\right)\quad\text{as }\lambda\to\infty,$$
where the term $o\left(r^n\right)$ is independent of $x$.
\item  If in addition $r^{-1}=o(\lambda)$, i.e. $r/r_\Planck\to\infty$, then the variance
$$\Var\left(\int_{B(x,r)}|u_\lambda|^2\,d\Vol\right)=o\left(r^{2n}\right)\quad\text{as }\lambda\to\infty,$$
where the term $o\left(r^{2n}\right)$ is independent of $x$.
\end{enumerate}
\end{thm}

\begin{rmk}
In particular, if we choose $W=1$, then $u_\lambda$ in Theorem \ref{thm:sserandomfixed} is the asymptotically fixed frequency ensemble considered in Zelditch \cite{Z3}. In this case, Theorem \ref{thm:sserandomfixed} states that such ensembles are equidistributed at all scales $r$ such that $r/r_\Planck\to\infty$, on manifolds that satisfy the non self-focal condition.
\end{rmk}

Theorems \ref{thm:sserandom} and \ref{thm:sserandomfixed} give sufficient conditions such that a typical random wave $u_{\lambda,a}$ equidistributes on the ball $B(x,r)$ for $x\in\M$ and $r=r(\lambda)$ such that $r/r_\Planck\to\infty$. That is, the probability (i.e. measure in the probability space $\s^{N_W(\lambda)-1}$) that $u_{\lambda,a}$ \textit{does not} equidistribute on $B(x,r)$ decays in $\lambda$. 

However, neither theorem provides a quantitative estimate of the decay of such probability, nor do they conclude equidistribution of $u_{\lambda,a}$ on the whole manifold, i.e. on balls $B(x,r)$ for all $x\in\M$. We now address these two problems and provide a quantitative estimate of the probability for small scale equidistribution of random waves $u_{\lambda,a}$ \textit{uniformly} on $\M$. We measure the deviation from equidistribution with an $o(1)$ order function. That is, we say $m(\lambda)$ is an $o(1)$ order function if $m(\lambda):\R^{+}\to\R^{+}$ and $m(\lambda)\to0$ as $\lambda\to\infty$. 

For the random variable $a\in\s^{N_{W}(\lambda)-1}$, we say that $u_{\lambda,a}$ equidistributes at scale $r=r(\lambda)$ to order $m(\lambda)$ \textit{uniformly} on $\M$ if
\begin{equation}\label{eq:SSEonM}
\sup_{x\in\M}\left|\int_{B(x,r)}|u_{\lambda,a}|^2\,d\Vol-\frac{\Vol(B(x,r))}{\Vol(\M)}\right|\leq{}r^nm(\lambda).
\end{equation} 
Therefore, the introduction of $m(\lambda)$ quantitatively characterizes the reminder term in the equidistribution statements. In general, to make the remainder smaller (and equidistribution better), we need to pick larger scale balls.

We define the exceptional set $S_{r}(m)$ as the set of point in $\s^{N_{W}(\lambda)-1}$ where \eqref{eq:SSEonM} fails.
\begin{defn}
The exceptional set $S_{r}(m)$ is given by
$$S_{r}(m)=\left\{a\in \s^{N_{W}(\lambda)-1}:\exists x\in\M\text{ such that }\left|\int_{B(x,r)}|u_{\lambda,a}|^2\,d\Vol-\frac{\Vol(B(x,r))}{\Vol(\M)}\right|\geq{}r^nm(\lambda)\right\}.$$
\end{defn}

In the following theorem, we prove that the exceptional set $S_{r}(m)$ has exponentially small measure in $\s^{N_{W}(\lambda)-1}$ at certain scales that are larger than the ones in Theorems \ref{thm:sserandom} and \ref{thm:sserandomfixed}.
 
\begin{thm}[Uniform equidistribution of random waves at small scales]\label{thm:SSEonM}
There exist constants $c,\alpha,K>0$ depending only on $\M$ such that the following statements are true.
\begin{enumerate}[(i).]
\item Suppose $W=W(\lambda)$ such that $1\le W\le\lambda$ and $W\to\infty$ as $\lambda\to\infty$. Set
$$r_1=W^{-\frac{1}{2n}}\lambda^{-\frac{n-1}{2n}}.$$
Let $m(\lambda)\geq{}KW^{-1}$ and $r=r(\lambda)$ such that 
$$\max\left\{W^{-1},\alpha\log(\lambda)^{\frac{1}{2n}}r_{1}m(\lambda)^{-\frac{1}{n}}\right\}\le r\le\Inj\M.$$  
Then 
$$\mu_{N_W(\lambda)}(S_{r}(m))\le\exp\left(-\frac{cr^{2n}m(\lambda)^{2}}{r_1^{2n}}\right).$$
\item Suppose $W=W(\lambda)$ such that $1\le W\le\lambda$ and $W\to\infty$ as $\lambda\to\infty$. Set
$$r_2=W^{\frac{1}{2(n-1)}}\lambda^{-\frac12}.$$
Let $m(\lambda)\geq{}KW^{-1}$ and $r=r(\lambda)$ such that 
$$\alpha r_2\log(\lambda)^{\frac{1}{2(n-1)}}m(\lambda)^{-\frac{1}{n-1}}\le r\le W^{-1}.$$ 
Then 
$$\mu_{N_W(\lambda)}(S_{r}(m))\le\exp\left(-\frac{cr^{2(n-1)}m(\lambda)^{2}}{r_2^{2(n-1)}}\right).$$
\item Assume that the set of loop directions $\Loop_x$ is of measure zero in $S_x^*\M$ for all $x\in\M$. Suppose that $W>0$ is independent of $\lambda$. Then there exists an $o(1)$ order function $m(\lambda)$ such that if 
$$r\ge\alpha\log(\lambda)^{\frac{1}{2(n-1)}}\lambda^{-1/2}m(\lambda)^{-\frac{1}{n-1}},$$ 
then
$$\mu_{N_W(\lambda)}(S_{r}(m))\le\exp\left(-c\lambda^{n-1}r^{2(n-1)}m(\lambda)^{2}\right).$$
\end{enumerate}
\end{thm}

\begin{rmk}
The generality of Thoerem \ref{thm:SSEonM} can make it difficult to parse. In the special cases of the cut-off and asymptotically fixed ensembles, the results can be stated in a simpler fashion. That is, Theorem \ref{thm:SSEonM} concludes uniform equidistribution of $u_{\lambda,a}$ at scales approaching $\lambda^{-1/2}$ except an exponentially small set $S\subset\s^{N_{W}(\lambda)-1}$. More precisely,
\begin{enumerate}[(1).]
\item If $W=\lambda$, then $u_{\lambda,a}$ is the cut-off ensemble. By Case (i) of Theorem \ref{thm:SSEonM}, with $m(\lambda)=(\log\lambda)^{-\ve}$ for $\ve>0$ arbitrarily small, we have equidistribution up to scales $r$ such that $r\geq{}\alpha\log(\lambda)^{\frac{1+2\ve}{2n}}\lambda^{-1/2}$.
\item If $W=1$, then $u_{\lambda,a}$ is the cut-off ensemble. Assume further the loopset condition as in Case (iii) of Theorem \ref{thm:SSEonM}. Then for some $o(1)$ order function $m(\lambda)$, we have equidistribtuion up to scales $r$ such that $r\geq{}\alpha\log(\lambda)^{\frac{1}{2n}}m(\lambda)^{-\frac{1}{n-1}}\lambda^{-1/2}$.
\end{enumerate}
\end{rmk}

Throughout this paper, $A\lesssim B$ ($A\gtrsim B$) means $A\le cB$ ($A\ge cB$) for some constant $c$ depending only on the manifold; $A\approx B$ means $A\lesssim B$ and $B\lesssim A$; the constants $c$ and $C$ may vary from line to line.

\section{Preliminaries}\label{sec:pre}
A key technique in the study of randomisations of eigenfunctions is to reduce questions about the expectation or variance of a random variable to problems involving the spectral projection operator. (See Proposition \ref{prop:viaspec}.) On a compact manifold $\M$, let $\{e_j\}_{j=0}^\infty$ be a real-valued eigenbasis of the Laplacian $\Delta$ with eigenvalues $\lambda_j^2$. Then the spectral projection operator onto the space 
$$\Span_{\lambda_{j}\in[0,\lambda]}\{e_j\}$$
has the kernel 
$$E_{[0,\lambda]}(x,y)=\sum_{\lambda_{j}\in[0,\lambda]}e_{j}(x)e_{j}(y).$$
The highest order asymptotics of the kernel of $E_{[0,\lambda]}$ are well understood and there are a number of estimates linking the geometry of $\M$ to the behaviour of lower order terms. (See Theorems \ref{thm:spec} and \ref{thm:specImprov}.) In \S\ref{sec:spec}, we recall these spectral estimates of the Laplacian and their connection to underlying geometry. 

In \S \ref{sec:prob}, we discuss some probabilistic estimates including the Levy concentration of measure (Theorem \ref{thm:Levy}) from probability theory. It is this concentration of measure that we use to prove uniform equidistribution of random waves at small scales in Theorem \ref{thm:SSEonM}.

\subsection{Spectral estimates}\label{sec:spec}
Let $T^*\M=\{(x,\xi):x\in\M,\xi\in T_x^*\M\}$ be the cotangent bundle of $\M$ and $|\cdot|_x$ be the induced metric on the cotangent space $T_x^*\M$. We denote $\Exp_y$ the exponential map at $y\in\M$. Since $\Inj\M\ge1$, $\Exp_y(x)$ is diffeomophism if $d(x,y)$ is small enough. Here, $d$ denotes the Riemannian distance on $\M$. The next theorem from H\"ormander \cite{Ho} provides the estimates of the kernel $E_{[0,\lambda]}$.

\begin{thm}[Spectral projection kernel estimates]\label{thm:spec}
There is constant $d_0$ depending only on $\M$ such that if $d(x,y)<d_0$, then
$$E_{[0,\lambda]}(x,y)=\frac{1}{(2\pi)^{n}}\int_{|\xi|_{g_{y}}\le\lambda}e^{i\langle\Exp^{-1}_{y}(x),\xi\rangle}\,\frac{d\xi}{\sqrt{|g_{y}|}}+R(x,y,\lambda),$$
where $R(x,y,\lambda)=O(\lambda^{n-1})$ as $\lambda\to\infty$ uniformly for $x,y\in\M$ such that $d(x,y)<d_0$. 
\end{thm}

Letting $x=y$ in the above theorem, we immediately get the pointwise Weyl asymptotics as well as the Weyl asymptotics for the distribution of eigenvalues.
\begin{cor}[Pointwise Weyl asymptotics]\label{cor:Weyl}
We have that
\begin{equation}\label{eq:ptWeyl}
E_{[0,\lambda]}(x,x)=\sum_{\lambda_j\le\lambda}|e_j(x)|^2=c_n\lambda^n+R(\lambda,x),\quad\text{where }R(\lambda,x)=O(\lambda^{n-1})\text{ as }\lambda\to\infty
\end{equation}
uniformly for all $x\in\M$. Here, $c_n$ is the volume of the unit ball in $\R^n$. Moreover, let $N(\lambda):=\#\{j:\lambda_j\le\lambda\}$. Then
\begin{equation}\label{eq:Weyl}
N(\lambda)=c_n\Vol(\M)\lambda^n+R(\lambda),\quad\text{where }R(\lambda)=O(\lambda^{n-1})\text{ as }\lambda\to\infty.
\end{equation}
\end{cor}

The remainder term estimate $R(\lambda,x)=O(\lambda^{n-1})$ in \eqref{eq:ptWeyl} is sharp on the sphere $\s^n$. The sharp growth rate $\lambda^{n-1}$ is achieved at the poles of zonal harmonics on $\s^n$. See H\"ormander \cite[Section 6]{Ho}.

However, on some manifolds other than the sphere, the above estimates of $R(\lambda,x)$ and $R(\lambda)$ may be improved. Such improvements are related to the dynamical properties of the geodesic flow on $\M$. The geodesic flow $G_t$ is the Hamiltonian flow with Hamiltonian defined on $T^*\M$ as $H(x,\xi)=|\xi|_x^2$. The geodesic flow $G_t$ preserves the Liouville measure on $T^*\M$. Write the cosphere bundle $S^*\M=\{(x,\xi)\in T^*\M:|\xi|_x=1\}$. Then $G_t$ acts on $S^*\M$ by homogeneity and leaves the induced Liouville measure on $S^*\M$ invariant.

Denote the set of periodic geodesics on $S^*\M$ as
$$\Pi=\{(x,\xi)\in S^*\M:G_t(x,\xi)=(x,\xi)\text{ for some }t>0\}.$$
Duistermaat-Guillemin \cite{DG} proved that 
\begin{thm}[Improved Weyl asymptotics]\label{thm:WeylImprov}
Assume that the set of periodic geodesics $\Pi$ is of Liouville measure zero in $S^*\M$. Then
\begin{equation}\label{eq:WeylImprov}
N(\lambda)=c_n\Vol(\M)\lambda^n+R(\lambda),\quad\text{where }R(\lambda)=o(\lambda^{n-1})\text{ as }\lambda\to\infty.
\end{equation}
\end{thm}

To get the improvement of pointwise Weyl law, we need a pointwise dynamical condition on the geodesics that is similar to the one in Theorem \ref{thm:WeylImprov}. A geodesic loop through $x$ is a geodesic $L(t)$ parametrized by arclength so that for some $t_0>0$ such that $L(0)=L(t_0)=x$. Define the loop directions at $x$ as
$$\Loop_x:=\{\xi\in S^*\M:G_t(x,\xi)=(x,\eta)\text{ for some }t>0\text{ and }\eta\in S_x^*\M\}.$$
Canzani-Hanin \cite{CH} proved that 
\begin{thm}[Improved spectral projection estimate]\label{thm:specImprov}
Assume that $\Loop_x$ is of measure zero in $S_x^*\M$ for all $x\in\M$. Then
$$E_{[0,\lambda]}(x,y)=\frac{1}{(2\pi)^{n}}\int_{|\xi|_{g_{y}}<\lambda}e^{i\langle\Exp^{-1}_{y}(x),\xi\rangle}\,\frac{d\xi}{\sqrt{|g_{y}|}}+R(x,y,\lambda),$$
where $R(x,y,\lambda)=o(\lambda^{n-1})$ uniformly for all $x,y\in\M$. In particular, the pointwise Weyl asymptotic asserts that
\begin{equation}\label{eq:ptWeylImprov}
E_{[0,\lambda]}(x,x)=\sum_{\lambda_j\le\lambda}|e_j(x)|^2=c_n\lambda^n+R(\lambda,x),\quad\text{where }R(\lambda,x)=o(\lambda^{n-1})\text{ as }\lambda\to\infty
\end{equation}
uniformly for all $x\in\M$. 
\end{thm}

\begin{rmk}\hfill
\begin{enumerate}[(1).]
\item If $\Loop_x$ is of measure zero on $S_x^*\M$ for all $x\in\M$, then the set of periodic geodesics $\Pi$ is of Liouville measure zero on $S^*\M$. Hence, one has that $R(\lambda)=o(\lambda^{n-1})$ as $\lambda\to\infty$ as an immediate corollary of Theorem \ref{thm:specImprov} (one can also instead integrate \eqref{eq:ptWeylImprov} on $\M$ directly).

\item There is a long history of works investigating the relationship between the geometric condition of the manifold and the improved pointwise Weyl asymptotic \eqref{eq:ptWeylImprov} over \eqref{eq:Weyl}. See Safarov \cite{Sa}, Sogge-Zelditch \cite{SZ}, Sogge-Toth-Zelditch \cite{STZ}, and Canzani-Hanin \cite{CH} for more details.
\end{enumerate}
\end{rmk}

\subsection{Probabilistic estimates}\label{sec:prob}
We define in \eqref{eq:ulambda} the random waves $u=\sum a_je_j$ such that the random coefficients are chosen from the unit sphere with uniform probability measure. In this section, we gather some standard estimates of this probabilistic density.

We mention that there are different randomisation precedures (\cite{B, CI, Z3}) where $a_j$ are chosen as identical and independent variables with proper normalization, e.g. Gaussian random variables or Rademacher random variables. As noted by Zelditch \cite[Section 0.1]{Z3}, choosing the random variables from the unit sphere is more intuitive. In addition, on the spheres, the Levy concentration of measure (Theorem \ref{thm:Levy}) is crucial to establish the uniform equidistribution of random waves in Theorem \ref{thm:SSEonM}. 

Let $\s^{d-1}\subset\R^d$ be the $(d-1)$-$\dim$ unit sphere endowed with the uniform probability measure $\mu_d$. Write
$$u=\sum_{j=1}^da_js_j,\quad\text{where }a=(a_1,...,a_{d})\in\s^{d-1}\text{ and }s=(s_1,...,s_{d})\in\R^{d}.$$
Notice that
$$|u|>t\quad\text{if and only if}\quad|\langle(a_1,...,a_{d}),(s_1(x),...,s_{d}(x))\rangle_{\R^{d}}|>t.$$
We then have the following fact. See e.g. Burq-Lebeau \cite[Appendix A]{BL} for an elementary proof.
\begin{lemma}\label{lemma:spheredistr}
$$\mu_d(|u|>t)=\begin{cases}
\left(1-\frac{t^2}{|s|^2}\right)^{\frac{d}{2}-1} & \text{if }0\le t<|s|,\\
0 & \text{if }t\ge |s|,
\end{cases}$$
where $|s|$ is the length of $s=(s_1,...,s_{d})\in\R^{d}$.
\end{lemma}
Because $|a_j|$, $j=1,...,d$, has identical distribution for $a\in\s^{d-1}$, it is obvious (since $\sum_{j=1}^{d}|a_j|^2=|a|^2=1$ for $a\in\s^{d-1}$) that $\E(|a_j|^2)=d^{-1}$ . Using Lemma \ref{lemma:spheredistr}, one can directly compute the $p$-moment $\E(|a_j|^p)$ for $p\ge0$. Taking $s$ as the unit vector in the $j$-th axis of $\R^{d}$, we have that $\langle a,s\rangle=a_j$, therefore, 
\begin{eqnarray}
\E(|a_j|^p)&=&\int_{\s^{d-1}}|a_j|^p\,da.\nonumber\\
&=&\int_0^\infty t^{p-1}\mu_d(|a_j|>t)\,dt\nonumber\\
&=&\int_0^1t^{p-1}\left(1-t^2\right)^{\frac{d}{2}-1}\,dt\nonumber\\
&=&\frac12\beta\left(\frac{d}{2},\frac{p}{2}\right)\nonumber\\
&=&c_pd^{-\frac p2}+O_p\left(d^{-\frac p2-1}\right)\label{eq:pmoment}.
\end{eqnarray} 
Here, $\beta(\cdot,\cdot)$ is the beta function and $c_p$ is a constant that depends only on $p$.

In the estimation of the variance in Theorem \ref{thm:sserandom}, we also need the expectation of $\E(|a_i|^2|a_j|^2)$ for $i\ne j$. (See Section \ref{sec:sserandom}.) If $a_j$ are chosen as independent random variables such as in the Gaussian ensemble $\sum_{j=1}^da_je_j$, then we normalize $\E\left(\|u\|^2_{L^2}\right)=1$ by setting $\E(|a_j|^2)=d^{-1}$. Hence, it is straightforward to see that $\E(|a_i|^2|a_j|^2)=d^{-2}$ for $i\ne j$ from the independence of $a_i$ and $a_j$. More generally, for $1\le m\le d$ and any $j_1,...,j_m\in\{1,...,d\}$ distinct, we have that 
$$\E\left(|a_{j_1}|^2\cdots |a_{j_m}|^2\right)=d^{-m}.$$
In Lemma \ref{lem:expmom} below, we show that that when $a_j$ are chosen from the unit sphere $\s^{d-1}$, $\E\left(|a_{j_1}|^2\cdots |a_{j_m}|^2\right)=d^{-m}+O(d^{-m-1})$, which agrees with the independent variable case modulo lower order term. 

First, let $p=(p_1,...,p_{d})$ for $p_j\ge0$, $j=1,...,d$. Denote $|p|=p_1+\cdots+p_{d}$. Then by \eqref{eq:pmoment} and H\"older inequality, we have that
\begin{align}
\E\left(|a_{1}|^{p_{1}}\cdots|a_{d}|^{p_{d}}\right)&=\int_{\s^{d-1}}|a_{1}|^{p_1}\cdots|a_{d}|^{p_{d}}\,da\nonumber\\
&\le\left(\int_{\s^{d-1}}|a_{1}|^{|p|}\,da\right)^\frac{p_1}{|p|}\cdots\left(\int_{\s^{d-1}}|a_{d}|^{|p|}\,da\right)^\frac{p_{d}}{|p|}\nonumber\\
&=O\left(d^{-\frac{|p|}{2}}\right).\label{betas}
\end{align} 

We now prove the following fact.

\begin{lemma}\label{lem:expmom}
Let $1\le m\le d$ and $j_k\in\{1,...,d\}$, $k=1,...,m$, be distinct. Then
$$A_m:=\E\left(a_{j_1}^2\cdots a_{j_m}^2\right)=\frac{1}{d^{m}}\left(1+O_{m}\left(d^{-1}\right)\right).$$
\end{lemma}
\begin{proof}
Since $a\in\s^{d-1}$,
$$1=\left(\sum_{j=1}^{d}a_{j}^{2}\right)^m=\sum_{j_k\text{ are distinct}}a_{j_{1}}^{2}\cdots a_{j_{m}}^{2}+\sum_{\text{others}}a_{j_{1}}^{2}\cdots a_{j_{m}}^{2}.$$
The first summation on the right-hand-side has $d(d-1)\cdots(d-m+1)$ terms and the second summation has $O_m(d^{m-1})$ terms. Note that the expectations of all the terms in the first summation are identical. Hence, taking the expectation of both sides in the above equation, we have that
\begin{eqnarray*}
1&=&d(d-1)\cdots(d-m+1)A_m+\sum_{\text{others}}\E\left(a_{j_{1}}^{2}\cdots a_{j_{m}}^{2}\right)\\
&=&d(d-1)\cdots(d-m+1)A_m+O_{m}(d^{m-1})O(d^{-m}).
\end{eqnarray*}
where we use \eqref{betas} to estimates the expectations in the second sum. The lemma therefore follows.
\end{proof}

To establish the uniform equidistribution of random waves in Theorem \ref{thm:SSEonM}, we need to control the probability that a function deviates from the expectation. To this end, we use the principle of concentration of measure. It is here that the high dimensionality of the probability spaces we consider comes into play. Concentration of measure requires that a random variable $F(a)$ cannot take values away from its median too often. Exactly how close to the median depends on regularity properties of $F$. Let
$$\|F\|_\Lip:=\sup_{a\ne b}\frac{|F(a)-F(b)|}{\dist(a,b)},$$
where $\dist(\cdot,\cdot)$ is the geodesic distance on $\s^{d-1}$. A number $\Me(F)$ is said to be a median value of $F$ if
$$\mu_d(F\ge\Me(F))\ge\frac12\quad\text{and}\quad\mu_d(F\le\Me(F))\ge\frac12.$$
Levy concentration of measures \cite[Theorem 2.3, (1.10), and (1.12)]{Le} then asserts that a Lipschitz function on $\s^{d-1}$ is highly concentrated around its median value when its dimension is large.
\begin{thm}[Levy concentration of measures]\label{thm:Levy}
Consider a Lipschitz function $F$ on $\s^{d-1}$. Then for any $t>0$, we have that
$$\mu_d(|F-\Me(F)|>t)\le\exp\left(-\frac{(d-2)t^2}{2\|F\|_\Lip^2}\right).$$
\end{thm}

\section{Proofs of Theorems \ref{thm:sserandom} and \ref{thm:sserandomfixed}}\label{sec:sserandom}
In this section, we prove the small scale equidistribution results in Theorems \ref{thm:sserandom} and \ref{thm:sserandomfixed}. Recall that
$$\HH_{W}(\lambda)=\Span_{\lambda_{j}\in[\lambda-W,\lambda]}\{e_j\}\quad\text{and}\quad N_{W}(\lambda)=\dim\HH_W(\lambda).$$
We write the kernel of the spectral projection operator onto $\HH_W(\lambda)$ as 
$$E_{[\lambda-W,\lambda]}(x,y)=\sum_{\lambda_{j}\in[\lambda-W,\lambda]}e_j(x)e_j(y)\quad\text{for }x,y\in\M.$$ 
For $a\in\s^{N_{W}(\lambda)-1}$, let 
$$u_{\lambda,a}=\sum_{\lambda_{j}\in[\lambda-W,\lambda]}a_je_j\in\HH_W(\lambda)$$ 
be a random wave. In the following proposition, we reduce the estimates of expectation and variance in Theorems \ref{thm:sserandom} and \ref{thm:sserandomfixed} to integrals involving the spectral projection kernel $E_{[\lambda-W,\lambda]}$.

\begin{prop}[Expectation and variance of random waves via the spectral kernel]\label{prop:viaspec}
Let  $\Omega\subset\M$ be measurable. Write
\begin{equation}\label{Fdef}
F_\Omega(a)=\int_\Omega|u_{\lambda,a}(x)|^2\,dx.
\end{equation}
Then
$$(i).\quad\E(F_\Omega)=\frac{1}{N_W(\lambda)}\int_\Omega E_{[\lambda-W,\lambda]}(x,x)\,dx,$$
and
\begin{eqnarray*}
(ii).\ \Var(F_{\Omega})&=&\frac{2}{N_{W}(\lambda)^2}\left(1+O\left(\frac{1}{N_{W}(\lambda)}\right)\right)\int_{\Omega}\int_{\Omega}E^{2}_{[\lambda-W,\lambda]}(x,y)\,dxdy\\
&&+O\left(s_{\Omega}\frac{\E(F_{\Omega})}{N_{W}(\lambda)}\right)+O\left(\frac{\E(F_{\Omega})^2}{N_{W}(\lambda)}\right),
\end{eqnarray*}
where
\begin{equation}\label{eq:sOmega}
s_{\Omega}=\sup_{\lambda_j\in[\lambda-W,\lambda]}\int_{\Omega}e_{j}^{2}(x)\,dx.
\end{equation}
\end{prop}
\begin{proof}
For simplicity of notation, we renumber the eigenbasis of $\HH_W(\lambda)$ as $\{e_1,...,e_{N_W(\lambda)}\}$.

(i). To prove the expectation, denote 
$$e_{W,\lambda}(x)=|(e_1(x),...,e_{N_{W}(\lambda)}(x))|\quad\text{for }x\in\M,$$
that is, the length of the vector $(e_1(x),...,e_{N_{W}(\lambda)}(x))\in\R^{N_{W}(\lambda)}$. Then 
$$\E(F_\Omega)=\int_{\s^{N_{W}(\lambda)-1}}\int_\Omega\sum_{i,j=1}^{N_W(\lambda)}a_{i}a_{j}e_{i}(x)e_{j}(x)\,dxd\mu_{N_W(\lambda)}.$$
Recall that the $\mu_d$ is the uniform probability measure on the sphere $\s^{d-1}$. Since each of the $a_{i}$ has mean zero and $\E\left(a_{i}^{2}\right)=1/N_{W}(\lambda)$, the expectation follows 
$$\E(F_\Omega)=\frac{1}{N_{W}(\lambda)}\int_\Omega\sum_{i=1}^{N_W(\lambda)}|e_{i}(x)|^2\,dx=\frac{1}{N_W(\lambda)}\int_\Omega E_{[\lambda-W,\lambda]}(x,x)\,dx.$$

(ii). To prove the variance, we directly compute that
\begin{eqnarray*}
\Var(F_\Omega)&=&\int_{\s^{N_{W}(\lambda)-1}}\big|F_\Omega(a)-\E(F_\Omega)\big|^{2}\,d\mu_{{N_{W}(\lambda)}}\\
&=&\int_{\s^{N_{W}(\lambda)-1}}\left|\int_{\Omega}\sum_{i,j=1}^{N_W(\lambda)}a_{i}a_{j}e_{i}(x)e_{j}(x)\,dx-\E(F_\Omega)\right|^{2}\,d\mu_{N_{W}(\lambda)}\\
&=&\int_{\s^{N_{W}(\lambda)-1}}\int_{\Omega}\int_{\Omega}\sum_{i,j,k,l=1}^{N_W(\lambda)}a_{i}a_{j}a_{k}a_{l}e_{i}(x)e_{j}(x)e_{k}(y)e_{l}(y)\,dxdyd\mu_{N_{W}(\lambda)}\\
&&-2\E(F_\Omega)\int_{\s^{N_{W}(\lambda)-1}}\int_{\Omega}\sum_{i,j=1}^{N_W(\lambda)}a_{i}a_{j}e_{i}(x)e_{j}(x)\,dxd\mu_{N_{W}(\lambda)}+\E(F_\Omega)^{2}.
\end{eqnarray*}
Since each $a_{i}$ has mean zero on $\s^{N_{W}(\lambda)-1}$, any term containing odd powers of the $a_i$ is zero in expectation. So the terms with even powers remains only. Applying Lemma \ref{lem:expmom} with $m=2$ and $d=N_W(\lambda)$, we have that for $i\ne j$,
$$A_2=\E\left(a_i^2a_j^2\right)=\frac{1}{N_W(\lambda)^2}\left(1+O\left(N_W(\lambda)^{-1}\right)\right).$$
By \eqref{betas}, we have that $\E(a_i^4)=O\left(N_W(\lambda)^{-2}\right)$. Hence,
\begin{eqnarray*}
\Var(F_{\Omega})&=&\int_{\s^{N_{W}(\lambda)-1}}\int_{\Omega}\int_{\Omega}\sum_{i,j,k,l=1}^{N_W(\lambda)}a_{i}a_{j}a_{k}a_{l}e_{i}(x)e_{j}(x)e_{k}(y)e_{l}(y)\,dxdyd\mu_{N_{W}(\lambda)}-\E(F_\Omega)^{2}\\
&=&2\int_{\Omega}\int_{\Omega}\sum_{\substack{i,j=1\\ i\neq{}j}}^{N_W(\lambda)}A_{2}e_{i}(x)e_{i}(y)e_{j}(x)e_{j}(y)\,dxdy+\int_{\Omega}\int_{\Omega}\sum_{\substack{i,j=1\\ i\neq j}}^{N_W(\lambda)}A_{2}e_{i}^{2}(x)e_{j}^2(y)\,dxdy\\
&&+\int_{\Omega}\int_{\Omega}\sum_{i=1}^{N_W(\lambda)}\E(a_{i}^{4})e_{i}^{2}(x)e_{i}^{2}(y)dxdy-\E(F_{\Omega})^2\\
&=&\frac{2}{N_{W}(\lambda)^2}\left(1+O\left(\frac{1}{N_{W}(\lambda)}\right)\right)\int_{\Omega}\int_{\Omega}\sum_{i,j=1}^{N_W(\lambda)}e_{i}(x)e_{i}(y)e_{j}(x)e_{j}(y)\,dxdy\\
&&+\frac{1}{N_{W}(\lambda)^2}\left(1+O\left(\frac{1}{N_{W}(\lambda)}\right)\right)\int_{\Omega}\int_{\Omega}\sum_{i,j=1}^{N_W(\lambda)}e_{i}^{2}(x)e_{j}^{2}(y)dxdy-\E(F_\Omega)^{2}\\
&&+O\left(\frac{1}{N_{W}(\lambda)^2}\right)\int_{\Omega}\int_{\Omega}\sum_{i=1}^{N_W(\lambda)}e_{i}^{2}(x)e_{i}^{2}(y)\,dxdy\\
&=&\frac{2}{N_{W}(\lambda)^2}\left(1+O\left(\frac{1}{N_{W}(\lambda)}\right)\right)\int_{\Omega}\int_{\Omega}E^{2}_{[\lambda-W,\lambda]}(x,y)\,dxdy\\
&&+O\left(\frac{1}{N_{W}(\lambda)^2}\right)\int_{\Omega}\int_{\Omega}\sum_{i=1}^{N_W(\lambda)}e_{i}^{2}(x)e_{i}^{2}(y)\,dxdy\\
&&+O\left(\frac{1}{N_{W}(\lambda)^3}\right)\int_{\Omega}\int_{\Omega}\sum_{i,j=1}^{N_W(\lambda)}e_{i}^{2}(x)e_{j}^{2}(y)dxdy\\
&\le&\frac{2}{N_{W}(\lambda)^2}\left(1+O\left(\frac{1}{N_{W}(\lambda)}\right)\right)\int_{\Omega}\int_{\Omega}E^{2}_{[\lambda-W,\lambda]}(x,y)\,dxdy\\
&&+O\left(s_{\Omega}\frac{\E(F_{\Omega})}{N_{W}(\lambda)}\right)+O\left(\frac{\E(F_{\Omega})^2}{N_{W}(\lambda)}\right).
\end{eqnarray*}
Here, $s_\Omega$ is defined in \eqref{eq:sOmega} and we use the fact that
$$\int_{\Omega}\int_{\Omega}\sum_{i=1}^{N_W(\lambda)}e_{i}^{2}(x)e_{i}^{2}(y)\,dxdy\leq{}s_{\Omega}\int_{\Omega}\sum_{i=1}^{N_W(\lambda)}e_{i}^{2}(x)\,dx=s_{\Omega}N_{W}(\lambda)\E(F_{\Omega}).$$
\end{proof}

Now we prove the small scale equidistribution results of random waves in Theorem \ref{thm:sserandom}.

\begin{proof}[Proof of Theorem \ref{thm:sserandom}]
Since the spectral projection kernel
$$E_{[\lambda-W,\lambda]}=E_{[0,\lambda]}-E_{[0,\lambda-W)},$$
the spectral estimates of the Laplacian in Section \ref{sec:spec} apply.
 
(i). To prove the expectation we use the pointwise Weyl asymptotics in Corollary \ref{cor:Weyl} to see that
\begin{eqnarray*}
E_{[\lambda-W,\lambda]}(x,x)&=&E_{[0,\lambda]}(x,x)-E_{[0,\lambda-W)}(x,x)\\
&=&c_n\lambda^n+R(\lambda,x)-c_n(\lambda-W)^n-R(\lambda-W,x)\\
&=&nc_nW\lambda^{n-1}+O(\lambda^{n-1}).
\end{eqnarray*}
Also from the Weyl asmptotics of the eigenvalues 
\begin{equation}\label{eq:NWlambda}
N_W(\lambda)=nc_nW\lambda^{n-1}\Vol(\M)+O(\lambda^{n-1}).
\end{equation}
Setting $\Omega=B(x,r)$ in (i) of Proposition \ref{prop:viaspec}, we conclude \eqref{eq:sserandomEx} in Theorem \ref{thm:sserandom}.

(ii). To prove the variance, we need the off-diagonal description of $E_{[\lambda-W,\lambda]}(x,y)$ of Theorem \ref{thm:spec}. For $x,y\in\M$ close enough, say $d(x,y)<d_0$ as in Theorem \ref{thm:spec},
\begin{equation}\label{eq:EW}
E_{[\lambda-W,\lambda]}(x,y)=\frac{1}{(2\pi)^{n}}\int_{\lambda-W\le|\xi|_{g_y}\le\lambda}e^{i\langle\Exp^{-1}_{y}(x),\xi\rangle}\,\frac{d\xi}{\sqrt{|g_{y}|}}+R(x,y,\lambda)-R(x,y,\lambda-W).
\end{equation}
Here, for notational simplicity, we write
$$R(\lambda,W)=\sup_{d(x,y)<d_0}|R(x,y,\lambda)-R(x,y,\lambda-W)|.$$
Hence, by Theorem \ref{thm:spec}
\begin{equation}\label{eq:RlambdaW}
R(\lambda,W)=O(\lambda^{n-1}).
\end{equation}
Since $x$ and $y$ are close, we can assume that they are in the same coordinate patch. Indeed, we may assume that $y$ is the centre of that patch and $g_{y}=\Id$. Therefore, the integral in \eqref{eq:EW} becomes
$$\frac{1}{(2\pi)^{n}}\int_{\lambda-W<|\xi|\le\lambda}e^{i\langle\Exp^{-1}_{y}(x),\xi\rangle}\,{d\xi}.$$
Note that the inner product here is understood by associating $\Exp^{-1}_{y}(x)$ with an element of $\R^{n}$ so effectively what we need to calculate is
$$\frac{1}{(2\pi)^{n}}\int_{\lambda-W\le|\xi|\le\lambda}e^{i\langle z,\xi\rangle}\,d\xi=\frac{1}{(2\pi)^{n}}\int_{\lambda-W}^\lambda\int_{|\xi|=\rho}e^{i\langle z,\xi\rangle}\,d\xi d\rho.$$
That is, we need to take the inverse Fourier transform of the surface measure of the sphere of radius $\rho$. This is a classical problem from harmonic analysis and can be computed by stationary phase to give
$$\left|\int_{|\xi|=\rho}e^{i\langle z,\xi\rangle}\,{d\xi}\right|\le c\rho^{n-1}(1+\rho|z|)^{-\frac{n-1}{2}},$$
in which $c$ depends only on $n$. See e.g. Sogge \cite[Section 1.2]{So3}. Therefore,
$$\frac{1}{(2\pi)^{n}}\int_{\lambda-W<|\xi|\le\lambda}e^{i\langle\Exp^{-1}_{y}(x),\xi\rangle}\,{d\xi}\le\begin{cases}
cW\lambda^{n-1}, & \text{if }0\le|x-y|\le\lambda^{-1};\\
cW\lambda^\frac{n-1}{2}|x-y|^{-\frac{n-1}{2}}, & \text{if }\lambda^{-1}\le|x-y|\le1.
\end{cases}$$
Taking $\Omega=B(x_0,r)$ in (ii) of Proposition \ref{prop:viaspec}, we compute that
\begin{align*}
\int_{B(x_{0},r)}\int_{B(x_{0},r)}E_{[\lambda-W,\lambda]}^{2}(x,y)\,dxdy&\le c\int_{B(x_{0},r)}\int_{B(x_{0},r)\cap B(x,\lambda^{-1})}W^2\lambda^{2(n-1)}\,dydx\\
&\quad+c\int_{B(x_{0},r)}\int_{B(x_{0},r)\setminus B(x,\lambda^{-1})}W^2\lambda^{n-1}|x-y|^{-(n-1)}\,dydx\\
&\quad+\int_{B(x_{0},r)}\int_{B(x_{0},r)}R(\lambda,W)^{2}\,dxdy\\
&\le cW^2\lambda^{n-2}\Vol(B(x_0,r))+cW^2\lambda^{n-1}r\Vol(B(x_0,r))\\
&\quad+R(\lambda,W)^{2}\Vol(B(x_0,r))^2.
\end{align*}
Combining with \eqref{eq:NWlambda},
\begin{eqnarray}
\Var(F_{B(x_{0},r)})&=&\frac{2}{N_{W}(\lambda)^{2}}\left(1+O\left(\frac{1}{N_{W}(\lambda)}\right)\right)\int_{B(x_{0},r)}\int_{B(x_{0},r)}E_{[\lambda-W,\lambda]}^{2}(x,y)\,dxdy\nonumber\\
&&+ O\left(s_{B(x_{0},r)}\frac{\E(F_{B(x_{0},r)})}{N_{W}(\lambda)}\right)+O\left(\frac{\E(F_{B(x_{0},r)})^2}{N_{W}(\lambda)}\right)\nonumber\\
&\le&c\lambda^{-n}\Vol(B(x_0,r))+c\lambda^{-(n-1)}r\Vol(B(x_0,r))\nonumber\\
&&+cW^{-2}\lambda^{-2(n-1)}R(\lambda,W)^{2}\Vol(B(x_0,r))^2\nonumber\\
&&+O\left(s_{B(x_{0},r)}\frac{\E(F_{B(x_{0},r})}{N_{W}(\lambda)}\right)+O\left(\frac{\E(F_{B(x_{0},r)})^2}{N_{W}(\lambda)}\right).\label{eq:VarFx0}
\end{eqnarray}
If $r^{-1}=o(\lambda)$, then $\lambda^{-n}=o(r^n)$ and so
$$\lambda^{-n}\Vol(B(x_0,r))=o\left(r^{2n}\right)\quad\text{and}\quad\lambda^{-(n-1)}r\Vol(B(x_0,r))=o\left(r^{2n}\right).$$
By Lemma \ref{lemma:worstballs}, we have that 
$$s_{B(x_{0},r)}\leq 
\begin{cases}
CWr & \text{if }\lambda^{-1}\leq{}r\leq{}W^{-1}\\
1 & \text{if }W^{-1}\leq{}r\leq\Inj(\M).\end{cases}$$
So in either of the above cases, 
$$s_{B(x_{0},r)}\frac{\E(F_{B(x_{0},r)})}{N_{W}(\lambda)}=o(r^{2n}).$$
Therefore, (ii) in Theorem \ref{thm:sserandom} follows from the fact that $R(\lambda,W)=O(\lambda^{n-1})$ in  \eqref{eq:RlambdaW}.
\end{proof}

We next provide a short proof of Theorem \ref{thm:sserandomfixed} in the case of fixed window length with the help of Theorem \ref{thm:specImprov}.

\begin{proof}[Proof of Theorem \ref{thm:sserandomfixed}]\hfill

(i). To prove the expectation, we use the improved pointwise Weyl asymptotics in Theorem \ref{thm:specImprov} to see that
\begin{eqnarray*}
E_{[\lambda-W,\lambda]}(x,x)&=&E_{[0,\lambda]}(x,x)-E_{[0,\lambda-W]}(x,x)\\
&=&c_n\lambda^n+R(\lambda,x)-c_n(\lambda-W)^n-R(\lambda-W,x)\\
&=&nc_nW\lambda^{n-1}+o(\lambda^{n-1}).
\end{eqnarray*}
We also have that
$$N_W(\lambda)=nc_nW\lambda^{n-1}\Vol(\M)+o(\lambda^{n-1}).$$
Taking $\Omega=B(x,r)$ in (i) of Proposition \ref{prop:viaspec}, we conclude (i) in Theorem \ref{thm:sserandomfixed}.

(ii) To prove the variance, we only need to notice that in the last term of \eqref{eq:VarFx0}, $R(\lambda,W)=o(\lambda^{n-1})$ from the improvement of remainder terms provided in Theorem \ref{thm:specImprov}.
\end{proof}

\section{Uniform equidistribution}\label{sec:uniform}
For a ball $B(x,r)$ with a \textit{fixed} $x\in\M$, the expectation and variance results in Theorems \ref{thm:sserandom} and \ref{thm:sserandomfixed} tell us that for scales $r=r(\lambda)$ such that $r/r_\Planck\to\infty$, the probability that the random wave $u_{\lambda,a}$ \textit{does not} equidistribute on $B(x,r)$ decays in $\lambda$.

In this section, we establish a quantitative estimate of the probability for small scale equidistribution of random waves \textit{uniformly} everywhere on the manifold, i.e. on balls all $B(x,r)$, $x\in\M$. 

First for a fixed $x\in\M$, we derive a quantitative estimate of the probability such that $u_{\lambda,a}$ does not equidistribute on $B(x,r)$. The main tool is from the Levy concentration of measure in Theorem \ref{thm:Levy}. We then use a covering argument to estimate the probability for equidistribution on all balls in $\M$.

Levy concentration of measure in Theorem \ref{thm:Levy} requires a Lipschitz norm estimate. To this end, we need a result from Sogge \cite[Section 4]{So1} that limits the maximum $L^{2}$ growth on a small ball. (It is actually proved for spectral clusters in \cite{So1}, which applies to combination of eigenfunctions in the same spectral window.)
\begin{lemma}\label{lemma:Sogge}
On a compact manifold $\M$, let $u=\sum_{\lambda_{j}\in[\lambda-1,\lambda]}a_{j}e_{j}$. Then for all $x\in\M$ and $\lambda^{-1}\le r\le\Inj\M$, we have that
\begin{equation}\int_{B(x,r)}|u|^2\,d\Vol\le cr\norm{u}^{2}_{L^{2}(\M)},\label{SoggeSS}\end{equation}
where $c>0$ depending only on $\M$.
\end{lemma}

Lemma \ref{lemma:Sogge} is of course an improvement on the trivial estimate $\int_{B(x,r)}|u|^2\,d\Vol\le1$. In fact, it is already sharp on $\s^n$, as the estimate is saturated by the zonal harmonics on balls centred at one of the poles. See Sogge \cite[Section 4]{So1} for more discussion. 

Since we consider window widths $W=W(\lambda)$ such that $1\le W\le\lambda$, we need to get an analogous estimate to \eqref{SoggeSS} for
$$\sum_{\lambda_{j}\in[\lambda-W,\lambda]}a_{j}e_{j}.$$
We are able to use Lemma \ref{lemma:Sogge} to obtain the necessary estimates. 

\begin{lemma}\label{lemma:worstballs}
Let $1\le W\le\lambda$. Suppose that for $a_j\in\R$,
$$u=\sum_{\lambda_{j}\in[\lambda-W,\lambda]}a_{j}e_{j}.$$
Then there exists a positive constant $c$ depending only on $\M$ such that
\begin{equation}\label{worstballs}
\int_{B(x,r)}|u|^{2}\,d\Vol\leq{}
\begin{cases}
cWr\norm{u}_{L^{2}(\M)}^{2}, & \text{if }\lambda^{-1}\leq{}r\leq{}W^{-1},\\
\norm{u}_{L^{2}(\M)}^{2}, & \text{if }W^{-1}\leq{}r\leq\Inj\M.
\end{cases}
\end{equation}
\end{lemma}

\begin{proof}
The inequalities in \eqref{worstballs} are trivially true if $r\ge W^{-1}$ so $Wr>1$. Therefore, we may assume $\lambda^{-1}\leq{}r\leq{}W^{-1}$. Also if $W\in[\lambda/2,\lambda]$, then
$$Wr\ge\frac\lambda2\cdot\lambda^{-1}\ge\frac12.$$
Therefore  \eqref{worstballs} holds with $c=2$.

Now we assume that $W\leq{}\lambda/2$. Write
$$u=\sum_{0\le k\le W-1}u_{k},\quad\text{where}\quad u_{k}=\sum_{\lambda_{j}\in{}[\lambda-W+k,\lambda-W+k+1]}a_{j}e_{j}.$$
Note that each $u_{k}$ is a fixed window spectral cluster at frequency $\mu_{k}=\lambda-W+k+1>\lambda/2$ so we may apply Lemma \ref{lemma:Sogge} to each of the $u_{k}$ separately. Thus,
$$\int_{B(x_{0},r)}|u|^{2}\,d\Vol=\sum_{0\le m,k\le W-1}\int_{B(x_{0},r)}u_{k}(x)u_{(k+m)_{W-1}}(x)\,d\Vol,$$
where
$$(k+m)_{W-1}=k+m\;\;\text{mod }(W-1).$$
Applying Lemma \ref{lemma:Sogge} and the Cauchy-Schwartz inequality, we have that
\begin{align*}
\int_{B(x_{0},r)}|u|^{2}\,d\Vol&\lesssim r\sum_{0\le m,k\le W-1}\norm{u_{k}}_{L^{2}(\M)}\norm{u_{(k+m)_{W-1}}}_{L^{2}(\M)}\\
&\lesssim r\sum_{0\le m\le W-1}\left(\sum_{0\le k\le W-1}\norm{u_{k}}_{L^{2}(\M)}^{2}\right)^{1/2}\left(\sum_{0\le k\le W-1}\norm{u_{(k+m)_{W-1}}}_{L^{2}(\M)}^{2}\right)^{1/2}\\
&\lesssim rW\norm{u}_{L^{2}(\M)}^{2}.
\end{align*}
\end{proof}

\begin{rmk}
It turns out that the above simple estimates are sharp. Spectral clusters $u$ of window width $W$ in Lemma \ref{lemma:worstballs} are spectral cases of approximate eigenfunctions with $L^{2}$ error no greater than $W\lambda$.  That is,
$$\norm{(\Delta-\lambda^{2})u}_{L^{2}(\M)}\lesssim{}W\lambda\norm{u}_{L^{2}(\M)}.$$
Such functions can localize in one ball with radius $r=W^{-1}$. See e.g. Tacy \cite{T}.
\end{rmk}

We now prove the following estimate on the Lipschitz norm of $F_{B(x,r)}$ defined in \eqref{Fdef}.

\begin{prop}\label{prop:Lipatx}
There exists a positive constant $c$ depending only on $\M$ such that
$$\|F_{B(x,r)}\|_\Lip\leq\begin{cases}
crW, & \text{if }\lambda^{-1}\leq{}r\leq{}W^{-1};\\
c, & \text{if }W^{-1}\leq{}r\leq{}\Inj\M.\end{cases}$$ 
\end{prop}

\begin{proof}
Given $u,v\in\HH_W(\lambda)$, let
$$u=\sum_{j=1}^{N_{W}(\lambda)}a_je_j\quad\text{and}\quad v=\sum_{j=1}^{N_{W}(\lambda)}b_je_j,$$
where $a=(a_1,...,a_{N_{W}(\lambda)})$ and $b=(b_1,...,b_{N_{W}(\lambda)})$ are in $\s^{N_{W}(\lambda)-1}$.
We have that
\begin{eqnarray*}
\left|F_{B(x,r)}(a)-F_{B(x,r)}(b)\right|&=&\left|\int_{B(x,r)}|u(y)|^2\,dx-\int_{B(x,r)}|v(y)|^2\,dy\right|\\
&\le&\int_{B(x,r)}\left|u(y)^2-v(y)^2\right|\,dx\\
&=&\int_{B(x,r)}|u(y)-v(y)||u(y)+v(y)|\,dy\\
&\le&\left(\int_{B(x,r)}|u(y)-v(y)|^2\,dy\right)^\frac12\left(\int_{B(x,r)}|u(y)+v(y)|^2\,dy\right)^\frac12\\
&\le&\norm{u-v}_{L^{2}(B(x,r))}\norm{u+v}_{L^{2}(B(x,r))}.
\end{eqnarray*}
If $\lambda^{-1}\le r\leq{}W^{-1}$, then by applying Lemma \ref{lemma:worstballs}, we obtain that
$$\norm{u-v}_{L^{2}(B(x,r))}\le cr^{1/2}W^{1/2}\norm{u-v}_{L^{2}(\M)},$$
and
$$\norm{u+v}_{L^{2}(B(x,r))}\le cr^{1/2}W^{1/2}\norm{u+v}_{L^{2}(\M)}.$$
It then follows that
\begin{equation}\label{eq:Fab}
\left|F_{B(x,r)}(a)-F_{B(x,r)}(b)\right|\le crW\norm{u-v}_{L^{2}(\M)}\norm{u+v}_{L^{2}(\M)}\le crW\norm{u-v}_{L^{2}(\M)},\end{equation}
since $\norm{u+v}_{L^{2}(\M)}\le\norm{u}_{L^2(\M)}+\norm{v}_{L^{2}(\M)}=2$. From
$$u(x)-v(x)=\sum_{j=1}^{N_{W}(\lambda)}(a_j-b_j)e_j(x),$$
we also have that
\begin{equation}\label{eq:distab}
\norm{u-v}_{L^{2}(\M)}=|a-b|\approx\dist(a,b).
\end{equation}
Here, $|a-b|$ is the distance of $a$ and $b$ in $\R^{N_W(\lambda)}$ while $\dist(a,b)$ is the distance of $a$ and $b$ on $\s^{N_{W}(\lambda)-1}$. Putting \eqref{eq:distab} together with \eqref{eq:Fab}, 
$$\|F_{B(x,r)}\|_\Lip=\sup_{a,b\in\s^{N_{W}(\lambda)-1},a\ne b}\frac{\left|F_{B(x,r)}(a)-F_{B(x,r)}(b)\right|}{\dist(a,b)}\le crW.$$
If $W^{-1}\le r\le\Inj\M$,  we use the trivial estimates that
$$\norm{u-v}_{L^{2}(B(x,r))}\le\norm{u-v}_{L^{2}(\M)}\quad\text{and}\quad\norm{u+v}_{L^{2}(B(x,r))}\le\norm{u+v}_{L^{2}(\M)}.$$
Thus,
$$\left|F_{B(x,r)}(a)-F_{B(x,r)}(b)\right|\le\norm{u-v}_{L^{2}(\M)}\norm{u+v}_{L^{2}(\M)}\le c\norm{u-v}_{L^{2}(\M)}.$$
In the view of \eqref{eq:distab} again, the Lipschitz norm of $F_{B(x,r)}$ when $W^{-1}\le r\le\Inj\M$ follows.
\end{proof}

We can now use Levy concentration of measure to control the probability that for a fixed $x$, $F_{B(x,r)}$ deviates from $\Vol(B)/\Vol(\M)$.
\begin{prop}\label{prop:sseatx}
Let $m(\lambda)$ be an $o(1)$ order function (i.e. $m:\R^{+}\to\R^{+}$ and $m(\lambda)\to0$ as $\lambda\to\infty$). For $x\in\M$, denote
\begin{equation}\label{eq:Srlx}
S_{r,x}(m)=\left\{a\in\s^{N_{W}(\lambda)-1}:\left|\int_{B(x,r)}|u_{\lambda,a}|^2\,d\Vol-\frac{\Vol(B(x,r))}{\Vol(\M)}\right|\geq r^{n}m(\lambda)\right\}.
\end{equation}
For some $c$ and $K$ depending only on $\M$, the following statements are true.
\begin{enumerate}[(i).]
\item Suppose $W=W(\lambda)$ such that $1\le W\le\lambda$ and $W\to\infty$ as $\lambda\to\infty$.  Set
$$r_1=W^{-\frac{1}{2n}}\lambda^{-\frac{n-1}{2n}}.$$
Then for all $m(\lambda)\geq KW^{-1}$ and $r=r(\lambda)$ such that 
$$\max\left\{W^{-1},r_{1}m(\lambda)^{-\frac{1}{n}}\right\}\leq r\leq{}\Inj(\M),$$ 
we have that 
$$\mu_{N_W(\lambda)}\left(S_{r,x}(m)\right)\le\exp\left(-\frac{cr^{2n}m(\lambda)^{2}}{r_1^{2n}}\right).$$
\item Suppose $W=W(\lambda)$ such that $1\le W\le\lambda$ and $W\to\infty$ as $\lambda\to\infty$. Set
$$r_2=W^{\frac{1}{2(n-1)}}\lambda^{-\frac12}.$$
Then for all $m(\lambda)\geq{}KW^{-1}$ and $r=r(\lambda)$ such that 
$$r_{2}m(\lambda)^{-\frac{1}{n-1}}\le r\le W^{-1},$$ 
we have that
$$\mu_{N_W(\lambda)}\left(S_{r,x}(m)\right)\le\exp\left(-\frac{cr^{2(n-1)}m(\lambda)^{2}}{r_2^{2(n-1)}}\right).$$
\item Assume that the set of loop directions $\Loop_x$ is of measure zero in $S_x^*\M$ for all $x\in\M$. Suppose that $W>0$ is independent of $\lambda$. Then there exists some $m(\lambda)$ that is $o(1)$ as $\lambda\to\infty$ such that for any $r$ satisfying 
$$\lambda^{-1/2}m(\lambda)^{-\frac{1}{n-1}}\leq{}r\leq{}\Inj(\M),$$
we have 
$$\mu_{N_W(\lambda)}\left(S_{r,x}(m)\right)\le\exp\left(-C\lambda^{n-1}r^{2(n-1)}m(\lambda)^{2}\right).$$
\end{enumerate}
\end{prop}

\begin{proof}
First we recall from the expectation estimate \eqref{eq:sserandomEx} in Theorem \ref{thm:sserandom} that
$$\E\left(F_{B(x,r)}\right)=\E\left(\int_{B(x,r)}|u_\lambda|^2\,d\Vol\right)=\frac{\Vol(B(x,r))}{\Vol(\M)}\left[1+O\left(W^{-1}\right)\right],$$
in which the term $O\left(W^{-1}\right)$ is independent of $x\in\M$.

By Levy concentration of measures in Theorem \ref{thm:Levy}, we estimate the difference between the expectation and the median. 
\begin{eqnarray*}
|\E(F_{B(x,r)})-\Me(F_{B(x,r)})|&=&\left|\int_{\s^{N_W(\lambda)-1}}F_{B(x,r)}(a)\,da-\Me(F_{B(x,r)})\right|\nonumber\\
&\le&\int_{\s^{N_W(\lambda)-1}}\left|F_{B(x,r)}(a)-\Me(F_{B(x,r)})\right|\,da\\
&=&\int_0^\infty\mu_{N_W(\lambda)}\left(|F_{B(x,r)}(a)-\Me(F_{B(x,r)})|>t\right)\,dt\\
&\le&\int_0^\infty\exp\left(-\frac{(N_W(\lambda)-2)t^2}{\|F_{B(x,r)}\|_\Lip^2}\right)\,dt\\
&\le&\frac{c\|F_{B(x,r)}\|_\Lip}{N_W(\lambda)^\frac12}
\end{eqnarray*}
for some absolute constant $c>0$. Putting this together with the expectation \eqref{eq:sserandomEx}, we then have that
$$\Me(F_{B(x,r)})=\frac{\Vol(B(x,r))}{\Vol(\M)}+R_1+R_2,$$
where
$$R_1=O\left(W^{-1}\right)r^n\quad\text{and}\quad R_2=O\left(\frac{\|F_{B(x,r)}\|_\Lip}{N_W(\lambda)^\frac12}\right).$$
Now we divide into the three cases listed in the proposition. 

Case (i). Since $r\ge W^{-1}$, the second inequality of Lipschitz norm estimate in Proposition \ref{prop:Lipatx} applies. That is,
$$\|F_{B(x,r)}\|_\Lip\le c.$$
Hence as $r\ge r_1m(\lambda)^{-\frac{1}{n}}$,
\begin{eqnarray*}
R_2=O\left(\frac{\|F_{B(x,r)}\|_\Lip}{N_W(\lambda)^\frac12}\right)=O\left(N_W(\lambda)^{-\frac12}\right)=O\left(W^{-\frac12}\lambda^{-\frac{n-1}{2}}\right)=O(r_1^n)=O(r^{n}m(\lambda)).
\end{eqnarray*}
 Therefore,
\begin{equation}\label{eq:medianatx}
\Me(F_{B(x,r)})=\frac{\Vol(B(x,r))}{\Vol(\M)}+O\left(r^{n}(W^{-1}+m(\lambda)
)\right).
\end{equation}
Now we use the Levy concentration of measure to control deviance from the median. Define
\begin{equation}\label{eq:Srxtilde}
\widetilde S_{r,x}(m)=\left\{a\in\s^{N_{W}(\lambda)-1}:\left|\int_{B(x,r)}|u_{\lambda,a}|^{2}d\Vol-\Me(F_{B(x,r)})\right|\geq{}r^{n}m(\lambda)\right\}.
\end{equation}
By the Levy concentration of measure in Theorem \ref{thm:Levy},
\begin{align*}
\mu_{N_{W}(\lambda)}\left(\widetilde{S}_{r,x}\left(\frac m2\right)\right)&\leq \exp\left(-\frac{(N_{W}(\lambda)-2)r^{2n}m(\lambda)^{2}}{8\|F_{B(x,r)}\|^2_\Lip}\right)\\
&\leq \exp\left(-cW\lambda^{n-1}r^{2n}m(\lambda)^{2}\right)\\
&=\exp\left(-\frac{cr^{2n}m(\lambda)^{2}}{r^{2n}_1}\right),
\end{align*}
as $r_{1}=W^{-\frac{1}{2n}}\lambda^{-\frac{n-1}{2n}}$. 

Now suppose that $a\in S_{r,x}(m)$. Then
$$\left|\int |u_{\lambda,a}|^{2}d\Vol-\frac{\Vol(B(x,r))}{\Vol(\M)}\right|\geq{}r^{n}m(\lambda).$$
So combining with \eqref{eq:medianatx},
\begin{align*}
\left|\int_{B(x,r)}|u_{\lambda,a}|^{2}\,d\Vol-\Me(F_{B(x,r)})\right|&\ge\left|\int_{B(x,r)}|u_{\lambda,a}|^{2}\,d\Vol-\frac{\Vol(B(x,r))}{\Vol(\M)}\right|-\left|\frac{\Vol(B(x,r))}{\Vol(\M)}-\Me(F_{B(x,r)})\right|\\
&\ge r^{n}m(\lambda)+O\left(r^{n}(W^{-1}+m(\lambda))\right).
\end{align*}
If $m\ge KW^{-1}$ for some sufficiently large $K$, then
$$\left|\int_{B(x,r)}|u_{\lambda,a}|^{2}\,d\Vol-\Me(F_{B(x,r)})\right|\geq{}\frac{3r^{n}m(\lambda)}{4}+O\left(r^{n}m(\lambda)\right).$$
Since $m(\lambda)\to0+$ as $\lambda\to\infty$, we can conclude that 
$$\left|\int_{B(x,r)}|u_{\lambda,a}|^{2}\,d\Vol-\Me(F_{B(x,r)})\right|\geq{}\frac{r^{n}m(\lambda)}{2},$$
for sufficiently large $\lambda$. That is, $a\in\widetilde S_{r,x}(m/2)$. Therefore,
$$\mu_{N_{W}(\lambda)}(S_{r,x}(m))\leq\mu_{\s^{N_{W}(\lambda)-1}}\left(\widetilde{S}_{r,x}\left(\frac m2\right)\right)\leq\exp\left(-\frac{cr^{2n}m(\lambda)^{2}}{r_{1}^{2n}}\right).$$

Case (ii). The reasoning follows as Case (i) however the Lipschitz norm is different. Since $r\le W^{-1}$, the first inequality of Lipschitz norm estimate in Proposition \ref{prop:Lipatx} applies. That is,
$$\|F_{B(x,r)}\|_\Lip\le crW.$$
Hence as $r\ge r_2m(\lambda)^{-\frac{1}{n-1}}$,
\begin{eqnarray*}
R_2=O\left(\frac{\|F_{B(x,r)}\|_\Lip}{N_W(\lambda)^\frac12}\right)=O\left(\frac{crW}{N_W(\lambda)^\frac12}\right)=O\left(rW^{\frac12}\lambda^{-\frac{n-1}{2}}\right)=O(rr_2^{n-1})=O(r^{n}m(\lambda)).
\end{eqnarray*}
Note that the above equation is also independent of $x\in\M$. So again,
$$\Me(F_{B(x,r)})=\frac{\Vol(B(x,r))}{\Vol(\M)}+O(r^{n}(W^{-1}+m(\lambda))).$$
Then with $\widetilde{S}_{r,w}(m)$ defined as in \eqref{eq:Srxtilde}, we have that
\begin{align*}
\mu_{N_{W}(\lambda)}\left(\widetilde{S}_{r,x}\left(\frac m2\right)\right)&\leq{}\exp\left(-\frac{(N_{W}(\lambda)-2)r^{2n}m(\lambda)^{2}}{8\|F_{B(x,r)}\|_{\Lip}^{2}}\right)\\
&\leq\exp\left(-cW^{-1}\lambda^{n-1}r^{2(n-1)}m(\lambda)^{2}\right)\\
&=\exp\left(-\frac{cr^{2(n-1)}m(\lambda)^{2}}{r_{2}^{2(n-1)}}\right),
\end{align*}
as $r_{2}=W^{\frac{1}{2(n-1)}}\lambda^{-\frac{1}{2}}$. As in Case (i), $a\in S_{r,x}(m)$ implies that $a\in \widetilde{S}_{r,x}(m/2)$ so we also have
$$\mu_{N_{W}(\lambda)}(S_{r,x}(m))\leq\exp\left(-\frac{cr^{2(n-1)}m(\lambda)^{2}}{r_{2}^{2(n-1)}}\right).$$

Case (iii). Now we address the case where the window width is allowed to be fixed but we assume that $\M$ satisfies the loop set conditions. By Theorem \ref{thm:sserandomfixed}, there exists some $o(1)$ order function $m(\lambda)$ so that
$$\E\left(F_{B(x,r)}\right)=\frac{\Vol(B(x,r))}{\Vol(\M)}+r^{n}m(\lambda).$$
In this case as $W$ is fixed the first Lipschitz norm estimate in Proposition \ref{prop:Lipatx} applies. So as $r\ge\lambda^{-1/2}m(\lambda)^{-\frac{1}{n-1}}$,
$$\left|\E\left(F_{B(x,r)}\right)-\Me\left(F_{B(x,r)}\right)\right|\leq{}r\lambda^{-\frac{n-1}{2}}\leq{}r^{n}m(\lambda).$$
Therefore,
$$\Me\left(F_{B(x,r)}\right)=\frac{\Vol(B(x,r))}{\Vol(M)}+r^{n}m(\lambda)+O(r^{n}m(\lambda)).$$
As in Case (ii), we have that for this specific $m(\lambda)$,
$$\mu_{N_{W}(\lambda)}\left(\widetilde{S}_{r,x}\left(\frac m2\right)\right)\leq{}\exp\left(-c\lambda^{n-1}r^{2(n-1)}m(\lambda)^{2}\right).$$
Also by the reasoning of Case (i), if $a\in S_{r,x}(m)$ then $a\in \widetilde{S}_{r,x}(m/2)$ so
$$\mu_{N_{W}(\lambda)}(S_{r,x}(m))\leq{}\exp\left(-c\lambda^{n-1}r^{2(n-1)}m(\lambda)^{2}\right).$$
\end{proof}

To prove uniform equidistribution in Theorem \ref{thm:SSEonM}, we use a covering lemma that is similar to the one in Han \cite[Section 3.2]{Ha2}. 

\begin{lemma}\label{lemma:covering}
For any $d>0$ there exists a family of geodesic balls that covers $\M$:
$$\bigcup_{p=1}^{N_d}B(x_p,d)\supset\M\quad\text{with }N_d\le cd^{-n},$$ 
where $c>0$ depends only on $\M$.
\end{lemma}

Given a cover $\{B(x_p,d)\}$, note that for any $r\gg d$, $\{B(x_{p},r)\}$ remains a cover. Moreover, in the cover $\{B(x_{p},r)\}$, the centers $x_p$ of the balls are separated by distances $d\ll r$. This enables us to efficiently approximate the $L^{2}$ mass of $u$ on $B(x,r)$ by the $L^{2}$ mass on one of the $B(x_{p},r)$.  Then the set $S_{r}(m)$ of all $a\in\s^{N_{W}(\lambda)-1}$ for which equidistribution fails at some point is contained in the union of the $S_{r,x_{p}}(m)$, for which we have estimates from Proposition \ref{prop:sseatx}. See below for the details of the proof. 

\begin{proof}[Proof of Theorem \ref{thm:SSEonM}]
We begin by choosing a cover as in Lemma \ref{lemma:covering} with 
$$d=r\lambda^{-(n-1)}m(\lambda)^{2}.$$ 
Then for any $x\in\M$ there exists an $x_{p}$ so that $x\in B(x_p,d)$. We now approximate $\Vol(B(x,r))$ by $\Vol(B(x_p,r))$ and the $L^2$ mass of $u_{\lambda,a}$ in $B(x,r)$ by the one in $B(x_p,r)$, respectively. First,
\begin{align}
\left|\frac{\Vol(B(x,r))}{\Vol(\M)}-\frac{\Vol(B(x_{p},r)}{\Vol(\M)}\right|&\le c\Vol\left(B(x_p,r+d)\setminus B(x_p,r-d)\right)\nonumber\\
&\le cdr^{n-1}\nonumber\\
&\le cr^{n}\lambda^{-(n-1)}m(\lambda)^{2}.\label{approxvol}
\end{align}
Second, since $\norm{u_{\lambda,a}}_{L^{\infty}}\leq{}c\lambda^{\frac{n-1}{2}}$ (see e.g. \cite[Section 4.2]{So3}), we have that
\begin{align}
\left|\int_{B(x,r)}|u_{\lambda,a}|^{2}\,d\Vol-\int_{B(x_{p},r)}|u_{\lambda,a}|^{2}\,d\Vol\right|&\le c\norm{u_{\lambda,a}}_{L^{\infty}}^2\Vol\left(B(x_p,r+d)\setminus B(x_p,r-d)\right)\nonumber\\
&\le cr^{n}m(\lambda)^{2}.\label{approxL2}
\end{align}
We recall that 
$$S_{r}(m)=\left\{a\in\s^{N_{W}(\lambda)-1}:\exists x\in\M\text{ such that }\left|\int_{B(x,r)}|u_{\lambda,a}|^{2}\,d\Vol-\frac{\Vol(B(x,r)}{\Vol(\M)}\right|\ge r^{n}m(\lambda)\right\}.$$

Now suppose that $a\in S_r(m)$. Then by \eqref{approxvol} and \eqref{approxL2}, there exists $x_p$ such that
\begin{align*}
\left|\int_{B(x_{p},r)}|u_{\lambda,a}|^{2}\,d\Vol-\frac{\Vol(B(x_{p},r)}{\Vol(\M)}\right|&\ge\left|\int_{B(x,r)}|u_{\lambda,a}|^{2}\,d\Vol-\frac{\Vol(B(x,r)}{\Vol(\M)}\right|\\
&\qquad-\left|\frac{\Vol(B(x,r))}{\Vol(\M)}-\frac{\Vol(B(x_{p},r)}{\Vol(\M)}\right|\\
&\qquad-\left|\int_{B(x,r)}|u_{\lambda,a}|^{2}\,d\Vol-\int_{B(x_{p},r)}|u_{\lambda,a}|^{2}d\Vol\right|\\
&\ge r^{n}m(\lambda)-\left(cr^{n}\lambda^{-(n-1)}m(\lambda)^{2}+cr^{n}m(\lambda)^{2}\right)\\
&\ge r^{n}m(\lambda)+O(r^{n}m(\lambda)^{2})\\
&\ge r^{n}m(\lambda)/2,
\end{align*}
for sufficiently large $\lambda$ (since $m(\lambda)\to0+$ as $\lambda\to\infty$). It is then immediate that
\begin{equation}S_{r}(m)\in \bigcup_{p=1}^{N_d}S_{r,x_{p}}(m/2).\label{Scover}\end{equation}
Let us consider each case in the theorem separately.

Case (i). All the conditions in Proposition \ref{prop:sseatx} are satisfied. In fact, we have a stronger condition that $r\geq{}\alpha r_{1}m(\lambda)^{-\frac{1}{n}}\log(\lambda)^{\frac{1}{2n}}$. Therefore,
\begin{align*}
\mu_{N_{W}(\lambda)}\left(S_{r,x_{p}}\left(\frac m2\right)\right)&\leq{}\exp\left(-\frac{cr^{2n}m(\lambda)^{2}}{8 r_{1}^{2n}}\right)\\
&=\exp\left(-\frac{cr^{2n}m(\lambda)^{2}}{16 r_{1}^{2n}}\right)\exp\left(-\frac{cr^{2n}m(\lambda)^{2}}{16 r_{1}^{2n}}\right)\\
&\leq\exp\left(-\frac{c\alpha^{2n}\log\lambda}{16}\right)\exp\left(-\frac{cr^{2n}m(\lambda)^{2}}{16 r_{1}^{2n}}\right)\\
&\leq\lambda^{-\frac{c\alpha^{2n}}{16}}\exp\left(-\frac{cr^{2n}m(\lambda)^{2}}{16r_{1}^{2n}}\right).
\end{align*}
Now by \eqref{Scover} we have that
$$\mu_{N_{W}(\lambda)}\left(S_{r}(m)\right)\leq{}N_d\lambda^{-\frac{c\alpha^{2n}}{16}}\exp\left(-\frac{cr^{2n}m(\lambda)^{2}}{8r_{1}^{2n}}\right).$$
Further, notice that from our initial choice of $d$,
$$N_d=cd^{-n}=cr^{-n}\lambda^{n(n-1)}m(\lambda)^{-2n}\le c\lambda^{3n+n(n-1)}.$$
Here, we used the fact that $r,m(\lambda)\gtrsim W^{-1}\geq{}\lambda^{-1}$.  Since $N_{d}$ grows as a power of $\lambda$ we can  choose $\alpha$ large enough (depending only on the dimension of $\M$) so
$$N_d\lambda^{-\frac{c\alpha^{2n}}{16}}\le1.$$
We arrive at the desired estimate that
$$\mu_{N_{W}(\lambda)}(S_{r}(m))\leq\exp\left(-\frac{cr^{2n}m(\lambda)^{2}}{16 r_{1}^{2n}}\right).$$

Case (ii) follows the same reasoning with suitable adjustments (comparing $r$ to $r_{2}$) so we omit the proof. 

Case (iii).  Proposition \ref{prop:sseatx} tells us that there exists an $o(1)$ order function $m(\lambda)$ such that for any $x$,
$$\mu_{N_W(\lambda)}\left(S_{r,x}(m)\right)\le\exp\left(-c\lambda^{n-1}r^{2(n-1)}m(\lambda)^{2}\right).$$
Provided $m(\lambda)=O(\lambda^{-\beta})$ for some $\beta>0$, the argument remains the same as in Case  (i).\footnote{This is indeed true for all the cases such that the improvement in the pointwise Weyl law in Theorem \ref{thm:specImprov} is known. See e.g. \cite{CH}.} If the $m(\lambda)$ extracted from Proposition \ref{prop:sseatx} decays faster, we simply pick some $\widetilde{m}(\lambda)\leq{}m(\lambda)$ so that $\widetilde{m}(\lambda)=O(\lambda^{-\beta})$ for some $\beta>0$. Note that the results of Proposition \ref{prop:sseatx} will also hold for $\widetilde{m}(\lambda)$ and we complete the proof. 
\end{proof}

\end{document}